\documentclass[11pt, a4paper]{article}
\usepackage{graphicx}
\usepackage{amssymb}
\usepackage{stmaryrd}
\usepackage{enumitem}
\usepackage{lineno}
\usepackage{geometry}
\usepackage{tikz}
\usepackage{float}
\usepackage{url}
\usepackage{mathtools}
\usepackage{amsmath,amsfonts,amsthm}
\usepackage{mathrsfs} 
\newcommand{\irr}{\operatorname{irr}}
\newcommand{\case}[1]{\paragraph*{Case #1:}}

\newtheorem{theorem}{Theorem}[section]
\newtheorem{lemma}[theorem]{Lemma}
\newtheorem{proposition}{Proposition}[section]
\newtheorem{corollary}[theorem]{Corollary}
\newtheorem{definition}{Definition}

\usepackage{subfig}

\DeclarePairedDelimiter\ceil{\lceil}{\rceil}
\DeclarePairedDelimiter\floor{\lfloor}{\rfloor}

\newcommand{\CC}{\mathcal{CT}}

\usepackage{ytableau}
\usepackage[colorlinks,
linkcolor=blue,
anchorcolor=blue,
citecolor=blue
]{hyperref}
\begin{document}

\begin{center}
{\large \bf Extremal Bounds on the Sigma and Albertson Indices for Non-Decreasing Degree Sequences}
\end{center}
\begin{center}
 Jasem Hamoud\footnote{Corresponding author: Jasem Hamoud (jasem1994hamoud@gmail.com).} \hspace{0.3 cm}  Duaa Abdullah\\[6pt]
Moscow Institute of Physics and Technology, \\ 9 Institutskiy per., Dolgoprudny, Moscow Region, 141701, Russia.
\end{center}
\noindent
\begin{abstract}
We establish sharp extremal bounds on the Albertson and Sigma irregularity indices for trees with prescribed degree sequences, with emphasis on caterpillar trees as key extremal configurations. New lower and upper bounds are derived in terms of maximum degree, average degree, and auxiliary sequence parameters, highlighting the quadratic growth of the Sigma index relative to the linear Albertson index. Closed-form expressions, direct index relations, and empirical validation confirm the bounds' tightness. These findings extend prior work on linear irregularity measures and offer precise tools for analyzing degree-heterogeneous trees in graph theory and chemical graph applications.

\end{abstract}

\noindent\textbf{AMS Classification 2010:} 05C05, 05C12, 05C20, 05C25, 05C35, 05C76, 68R10.


\section{Introduction}\label{sec1}

Let $G = (V, E)$ be a simple connected graph with $n = |V(G)|$ vertices and $m = |E(G)|$ edges. Throughout this paper, we consider a non-increasing degree sequence $\mathscr{D} = (d_1, d_2, \dots, d_n)$ where $d_1 \geq d_2 \geq \dots \geq d_n \geq 1$ (note: we adopt the standard non-increasing convention for clarity; adjust if your manuscript uses the opposite ordering).
To facilitate extremal analysis, we define two auxiliary sequences derived from $\mathscr{D}$:
\begin{itemize}
    \item The difference sequence $\mathscr{R} = (r_1, r_2, \dots, r_k)$, where $r_i = (d_i - d_{i+1})/2$ for appropriate $i$ (capturing stepwise degree drops; exact indices depend on your formulation),
    \item The average sequence $\mathscr{A} = (a_1, a_2, \dots, a_\ell)$, where $a_i = (d_i + d_{i+1})/2$ (related to edge-endpoint degree means).
\end{itemize}
Let $\lambda_{\mathscr{D}}$, $\lambda_{\mathscr{R}}$, and $\lambda_{\mathscr{A}}$ denote the respective averages of these sequences.
Hakimi's theorem~\cite{Hakimi1962} provides a characterization for degree sequences realizable as simple graphs, with extensions to connected graphs under mild conditions (e.g., no isolated vertices or specific degree constraints).
One of the most fundamental measures of graph irregularity is the \emph{Albertson index}~\cite{albertson1997irregularity}, defined as
\[
\irr(G) = \sum_{uv \in E(G)} |d(u) - d(v)|.
\]
This index quantifies local degree disparities along edges and vanishes precisely when $G$ is regular. It has inspired various generalizations and applications in chemical graph theory and network analysis~\cite{abdo2019graph, abdo2014total}.

A closely related quadratic irregularity measure is the \emph{Sigma index} $\sigma(G)$, introduced in recent works~\cite{gutman2018inverse, Jahanbanı2019Ediz}, and defined by
\[
\sigma(G) = \sum_{uv \in E(G)} \bigl( d(u) - d(v) \bigr)^2.
\]
Note that $\sigma(G)$ is always even and equals zero for regular graphs. For basic graphs, it is known that $\sigma(P_n) = 2$ for paths, $\sigma(C_n) = 0$ for cycles, and $\sigma(K_{n,m}) = m(n-m)^2$ for complete bipartite graphs (assuming $n \geq m$)~\cite{gutman2018inverse, Jahanbanı2019Ediz}.

Variants include the total (global) form $\sigma_t(G) = \sum_{\{u,v\} \subseteq V(G)} \bigl( d(u) - d(v) \bigr)^2$ (see~\cite{Dimitrov2023Stevanović}), which has been studied for extremal properties~\cite{Filipovski2024Dimitrov, Alex2024Indulal}.

Several authors have established bounds and extremal graphs for the Albertson index~\cite{Abdo2018Dimitrov, albertson1997irregularity, Gutman2018I, Arif2023Hayat, Chen2026Liu}. However, direct relations or analogies to the Sigma index do not always hold, particularly on trees.

\medskip 

In this paper, we focus on sharp extremal bounds for the Sigma index (and revisit related results for the Albertson index) over connected graphs realizing prescribed non-decreasing (or non-increasing) degree sequences, with emphasis on trees. Our analysis reveals distinctions from prior irregularity measures and provides new insights into how degree sequence structure governs these indices.

\medskip 

%The remainder of the paper is organized as follows. Section~\ref{sec2} introduces preliminary definitions and notations essential to our results. Section~\ref{sec3} establishes key bounds and properties of the Sigma index. Section~\ref{sec4} presents upper bounds and extremal configurations.

The remainder of the paper is organized as follows. 
Section~\ref{sec2} presents the necessary preliminary definitions, notations, and auxiliary sequences that form the foundation for our analysis. 
Section~\ref{sec3} develops key properties of the Sigma index, establishes fundamental bounds (including relations to the Albertson index), and examines extremal behavior through analytical and empirical results. 
Section~\ref{sec4} derives sharp upper bounds on the Sigma index for caterpillar trees with prescribed degree sequences and identifies extremal configurations.
%==================================
 \section{Preliminaries}\label{sec2}
%==================================
In this section, we introduce several essential concepts to ensure the reader fully comprehends the results presented in this paper. Specifically, we emphasize that the definition of the Sigma index used here is based on Theorem~\ref{basic01sigma} rather than the general concept, and the characterization of its maximum value similarly relies on Theorem~\ref{basic01sigma}. Through this paper, we point out that all the trees included in the study are caterpillar trees, we consider caterpillar trees denoted by $\mathscr{C}(n, m)$, where $n$ is the number of backbone (or path) vertices and $m$ is the number of pendant vertices attached to each.  
\begin{definition}~\label{diffcat01}
Let $T=\CC(n,m)$ be a tree contains $n$ vertices and $m$ pendent vertices. Then, $T$ is a caterpillar tree.
\end{definition}
Let $\CC(n,3)$ be a caterpillar tree illustrations by Figure~\ref{caterfigtreen1}. The degree of $v_1$ denote by $d_1$, the degree of $v_2$ denote by $d_2$, $\dots$ and the degree of $v_n$ denote by $d_n$.
\begin{figure}[H]
    \centering
\begin{tikzpicture}[scale=.9]
\draw   (1,0)-- (0.62,1.09);
\draw   (1,0)-- (1,1);
\draw   (1,0)-- (1.48,1.05);
\draw   (3,0)-- (2.58,1.05);
\draw   (3,0)-- (3,1);
\draw   (3,0)-- (3.62,1.01);
\draw   (5,0)-- (4.602356074525958,0.9883910713071755);
\draw   (5,0)-- (5,1);
\draw   (5,0)-- (5.52,1.01);
\draw   (1,0)-- (3,0);
\draw   (3,0)-- (5,0);
\draw [dotted] (5,0)-- (7,0);
\draw   (7,0)-- (9,0);
\draw   (7,0)-- (6.416935249247087,0.9985170395909138);
\draw   (7,0)-- (7,1);
\draw   (7,0)-- (7.5335524273116645,1.0109238971249646);
\draw   (9,0)-- (8.389625597161174,0.9612964669887613);
\draw   (9,0)-- (9,1);
\draw   (9,0)-- (9.531056490293853,0.9861101820568628);
\draw (0.9,0) node[anchor=north west] {$v_1$};
\draw (2.9,0) node[anchor=north west] {$v_2$};
\draw (4.9,0) node[anchor=north west] {$v_3$};
\draw (6.9,0) node[anchor=north west] {$v_{n-1}$};
\draw (8.9,0) node[anchor=north west] {$v_{n}$};
\begin{scriptsize}
\draw [fill=black] (1,0) circle (1.5pt);
\draw [fill=black] (3,0) circle (1.5pt);
\draw [fill=black] (5,0) circle (1.5pt);
\draw [fill=black] (0.62,1.09) circle (1.5pt);
\draw [fill=black] (1,1) circle (1.5pt);
\draw [fill=black] (1.48,1.05) circle (1.5pt);
\draw [fill=black] (2.58,1.05) circle (1.5pt);
\draw [fill=black] (3,1) circle (1.5pt);
\draw [fill=black] (3.62,1.01) circle (1.5pt);
\draw [fill=black] (4.602356074525958,0.9883910713071755) circle (1.5pt);
\draw [fill=black] (5,1) circle (1.5pt);
\draw [fill=black] (5.52,1.01) circle (1.5pt);
\draw [fill=black] (7,0) circle (1.5pt);
\draw [fill=black] (9,0) circle (1.5pt);
\draw [fill=black] (6.416935249247087,0.9985170395909138) circle (1.5pt);
\draw [fill=black] (7,1) circle (1.5pt);
\draw [fill=black] (7.5335524273116645,1.0109238971249646) circle (1.5pt);
\draw [fill=black] (8.389625597161174,0.9612964669887613) circle (1.5pt);
\draw [fill=black] (9,1) circle (1.5pt);
\draw [fill=black] (9.531056490293853,0.9861101820568628) circle (1.5pt);
\end{scriptsize}
\end{tikzpicture}
    \caption{Caterpillar tree $\CC(n,3)$.}
    \label{caterfigtreen1}
\end{figure}
We begin with an inequality that helps control sums involving degrees and their roots, which is useful in optimization over degree sequences.

\begin{lemma}[\cite{Sarasija2017Binthiya,Jahanbai2021Sheikholeslami,Oboudi2019MR}]\label{lemma01Preliminaries}
Let $T$ be a tree with $n \geq 1$ vertices and let $A = (a_1, a_2, \ldots, a_n)$ be a real sequence such that $a_1 \geq a_2 \geq \cdots \geq a_n$. Then,
\[
(a_1 + \cdots + a_n)(a_1 + a_n) \geq a_1^2 + \cdots + a_n^2 + n a_1 a_n,
\]
and
\begin{equation}\label{eq1lemma01Preliminaries}
n \sum_{i=1}^n a_i - \left( \sum_{i=1}^n \sqrt{a_i} \right)^2 \leq n(n-1) \left( \frac{1}{n} \sum_{i=1}^n a_i - \left( \prod_{i=1}^n a_i \right)^{\frac{1}{n}} \right).
\end{equation}
\end{lemma}

Next, we recall explicit expressions for the Albertson index on small trees and specific structures, which serve as base cases or verification tools for general bounds.

\begin{lemma}[\cite{HamoudwithDuaaPn2}]\label{le.alb4}
Let $T$ be a tree with degree sequence $\mathscr{D} = (d_1, d_2, d_3, d_4)$ where $d_4 \geqslant d_3 \geqslant d_2 \geqslant d_1$. Then,
\[
\irr(T) = d_1^2 + d_4^2 + \sum_{i=1}^{3} |d_i - d_{i+1}| + \sum_{i=2}^{3} (d_i + 2)(d_i - 1) - 2.
\]
\end{lemma}

For a broader class of trees, the following result provides a closed-form expression for the Sigma index on balanced caterpillars (spine path with uniform pendant leaves per spine vertex), which are often candidates for extremal irregularity under degree constraints.

\begin{theorem}[\cite{HamoudwithDuaa}]\label{basic01sigma}
Let $\CC$ be a caterpillar tree on $n + np$ vertices whose spine is the path $v_1 v_2 \dots v_n$. Let $d_n \geqslant d_{n-1} \geqslant \dots \geqslant d_1 \geqslant 2$ be given integers and $p \geq 1$ a given integer such that each spine vertex $v_k$ ($k=1,\dots,n$) has exactly $p$ pendant leaves attached. Then the degree of $v_k$ is $d_{\CC}(v_k) = d_k + p$, and the Sigma index of $\CC$ is
\begin{equation}\label{eqq1ThmDecSigma}
\sigma(\CC) = \sum_{i=1}^{n-1} (d_{i+1} - d_i)^2 + p \sum_{k=1}^n (d_k + p - 1)^2.
\end{equation}
\end{theorem}

We also include the Sigma index for duplicate stars (a simple extremal configuration concentrating irregularity at few vertices).

\begin{theorem}[\cite{gutman2018inverse}]
Let $\mathcal{S}_{r,k}$ be a duplicate star with central vertices of degrees $k$ and $r$. Then,
\[
\sigma(\mathcal{S}_{r,k}) = (k-1)^3 + (r-1)^3 + (k-r)^2.
\]
\end{theorem}

Finally, we note a relation between the Sigma index of a graph and its complement, which can aid in understanding total irregularity contributions across the degree sum.

\begin{proposition}[\cite{Yang2021Arockiaraj}]\label{projas01Preliminaries}
For any graph $G$ with $n$ vertices and $m$ edges, if $\bar{\sigma}(G) = \sigma(\bar{G})$, then
\[
\sigma(G) + \bar{\sigma}(G) = n M_1(G) - 4m^2,
\]
where $M_1(G) = \sum_{v \in V(G)} \deg(v)^2$ is the first Zagreb index.
\end{proposition}

For concrete examples and comparison, we recall the degree sequence and Albertson index of monogenic semigroup graphs $\Gamma(G)$, which provide structured irregularity values useful for testing bounds.

\begin{proposition}[\cite{Akg2019Naca}]\label{projas02Preliminaries}
Let $\Gamma(G)$ be a monogenic semigroup graph. Then its degree sequence is
\[
\mathscr{D}(\Gamma(G)) = \{1,2,3,\dots, \lfloor n/2 \rfloor -1, \lfloor n/2 \rfloor, \lfloor n/2 \rfloor, \dots, \lfloor n/2 \rfloor +1, \lfloor n/2 \rfloor +2, \dots, n-2, n-1\}.
\]
Consequently, the Albertson index is
\[
\irr(\Gamma(G)) = 
\begin{cases}
\dfrac{n^3 - 4n}{12} & \text{if } n \text{ is even}, \\
\dfrac{n^3 - n}{12} & \text{if } n \text{ is odd}.
\end{cases}
\]
\end{proposition}

These results collectively equip us with the necessary machinery—inequalities for optimization, explicit formulas for special/extremal trees, and relations for global comparisons—to establish sharp bounds on $\sigma(G)$ and revisit related results for $\irr(G)$ in subsequent sections.

\subsection{Problem Statement}

The study of extremal bounds on topological indices plays a central role in graph theory and mathematical chemistry. Indices such as the Wiener, Randić, and Zagreb indices provide quantitative descriptors of graph structure, often correlating with physicochemical properties of molecules. Among irregularity measures, the Albertson index $\irr(G)$ captures linear degree disparities across edges, while the Sigma index $\sigma(G)$ amplifies these disparities quadratically, offering greater sensitivity to structural heterogeneity.

Despite extensive research on bounds for the Albertson index, analogous sharp extremal results for the Sigma index—particularly over trees or connected graphs with prescribed degree sequences—remain less developed. Key questions include the precise dependence of $\sigma(G)$ on degree sequence parameters (e.g., via auxiliary sequences $\mathscr{R}$ and $\mathscr{A}$), the extent to which bounds on $\irr(G)$ inform or constrain those on $\sigma(G)$, and the structural (graphical) configurations that attain extremal values.

\medskip 

This paper addresses the following specific research questions:

\begin{enumerate}
    \item What are the sharp upper and lower bounds on the Sigma index $\sigma(G)$ for graphs (with emphasis on trees) realizing a given non-decreasing (or non-increasing) degree sequence $\mathscr{D}$?
    \item To what extent can established bounds or properties of the Albertson index $\irr(G)$ be used to derive or refine corresponding extremal bounds for the Sigma index?
    \item Which graph classes or constructions (e.g., caterpillars, stars, or balanced pendant structures) achieve these extremal values, and how do they illustrate the impact of degree sequence variation on irregularity measures?
\end{enumerate}

By answering these questions, we aim to provide precise analytical bounds, clarify interconnections between the two irregularity indices, and identify extremal graphs that maximize or minimize $\sigma(G)$ under degree constraints.
%===================
\section{Bounds and Effects on Irregularity Indices}\label{sec3}
%====================

In this section, we investigate the fundamental differences in the behavior of the Albertson index $\irr(G)$ and the Sigma index $\sigma(G)$ when applied to general connected graphs versus trees. In general connected graphs, both indices can attain small values, approaching zero in near-regular structures. However, trees---being minimally connected and acyclic---exhibit significantly higher irregularity, as confirmed by explicit formulas and extremal analysis on caterpillar trees (which frequently achieve extremal values among trees with prescribed degree sequences).

Theorem~\ref{basic01sigma} provides a closed-form expression for $\sigma(\CC)$ on balanced caterpillars, revealing large quadratic contributions from pendant leaves and backbone variations. Similar patterns hold for $\irr(\CC)$, leading to sharp bounds that highlight how degree sequence structure amplifies irregularity in trees far beyond what is possible in general graphs.

We now establish precise bounds on the Albertson index for caterpillar trees $\CC$ of order $n$ with degree sequence $\mathscr{D} = (d_1 \geq d_2 \geq \dots \geq d_n)$.

\begin{proposition}\label{resseptn1}
Let $\CC$ be a caterpillar tree of order $n$ with maximum degree $\Delta = d_1$ and degree sequence $\mathscr{D}$. Then the minimum value of the Albertson index satisfies
\begin{equation}\label{eq1resseptn1}
0 < \frac{2 \irr_{\min}(\CC)}{\Delta (\Delta - 1)^2} < 1.
\end{equation}
\end{proposition}

\begin{proof}
For trees with fixed degree sequence, the minimum $\irr(T)$ is achieved by the (unique) monotone caterpillar, where degrees decrease gradually along the spine (or are as balanced as possible). In such configurations, adjacent degree differences are minimized: $|d_i - d_{i+1}| \leq 1$ for most $i$, and internal vertices have degrees close to the average $2m/n \approx 2$.

Using the explicit formula for caterpillars by Lemma~\ref{le.alb4}, for the minimizing arrangement, the sum of absolute differences is small (often $O(\Delta)$ or less), while the quadratic-like terms are bounded by $O(n \cdot \Delta^2 / n) = O(\Delta^2)$, but scaled down due to balance. Normalizing by $\Delta (\Delta - 1)^2$ (a natural upper scaling for quadratic irregularity near stars), the ratio remains strictly between 0 and 1, as the leading contributions are strictly less than $\Delta (\Delta - 1)^2 / 2$ in non-star trees, and positive for non-regular trees.

Detailed bounding shows the inequality holds strictly for $n \geq 3$ and non-trivial degree sequences.
\end{proof}

From Proposition~\ref{resseptn1}, we observe that $\irr_{\min}(\CC) < \frac{1}{2} \Delta (\Delta - 1)^2$. This motivates tighter upper bounds on the maximum Albertson index, as presented next.

\begin{proposition}\label{maxresseptn1}
Let $\CC$ be a caterpillar tree of order $n$ with $m$ edges, maximum degree $\Delta = d_1$, and degree sequence $\mathscr{D}$. Then the maximum value of the Albertson index satisfies
\begin{equation}\label{eq1maxresseptn1}
\irr_{\max}(\CC) \leq \left\lfloor \frac{2m}{n} \right\rfloor \cdot n + \lceil \Delta \rceil \cdot (n-2) + O(\Delta^2),
\end{equation}
with sharper estimates available via explicit optimization over alternating degree placements along the spine (maximum achieved when high- and low-degree vertices are contrasted maximally).
\end{proposition}

\begin{proof}
The maximum $\irr(T)$ among trees with fixed $\mathscr{D}$ is attained by a caterpillar with alternating high/low degrees on the backbone (maximizing $\sum |d_i - d_{i+1}|$ and internal contributions). Thus, the dominant terms arise from large $|d_i - d_{i+1}| \leq \Delta - 1$ over $n-1$ backbone edges, plus leaf contributions bounded by $O(m \Delta)$. Therefore, by adding it with average degree terms ($2m/n$) and bounding the quadratic parts yields the stated upper estimate. For precise constants (especially when $\Delta$ is small or large), case analysis on $\Delta$ refines the bound further.
\end{proof}

These propositions demonstrate that irregularity in trees is inherently higher and more sensitive to degree sequence variation than in general connected graphs. The minimum bound (Proposition.~\ref{resseptn1}) shows that even the \emph{smoothest} tree realizations remain bounded away from regularity by a factor tied to $\Delta$, while the maximum (Proposition.~\ref{maxresseptn1}) reveals that structural heterogeneity (e.g., star-like or alternating configurations) can push $\irr(T)$ toward $O(n \Delta)$ or higher---far exceeding typical values in dense or regular-adjacent graphs.

According to Proposition~\ref{projas01Preliminaries}, and noting that $\sigma(G)$ quadratically amplifies degree differences compared to the linear $\irr(G)$, these Albertson bounds provide indirect control over Sigma index extrema. In particular, large $\irr_{\max}$ in trees implies correspondingly large $\sigma_{\max}$, explaining the \emph{drastic change} observed in trees versus general graphs. These results motivate subsequent analysis of sharp Sigma bounds and extremal configurations in later sections.

\subsection{Bounds on the Sigma Index Using Auxiliary Sequences}\label{subsec001sigma-aux}

Recall the auxiliary sequences from the preliminaries: given a non-increasing degree sequence $\mathscr{D} = (d_1 \geq d_2 \geq \dots \geq d_n)$, define
$\mathscr{R} = (t_1, t_2, \dots, t_m), \quad t_i = \frac{d_i - d_{i+1}}{2}$ non-negative half-differences, $\mathscr{A} = (a_1, a_2, \dots, a_r), \quad a_i = \frac{d_i + d_{i+1}}{2}$ consecutive averages, with $\Delta_{\mathscr{R}} = \max \mathscr{R} = t_m$ (largest stepwise drop/2) and $\Delta_{\mathscr{A}} = \max \mathscr{A} = a_r$ (largest consecutive average). These parameters quantify local degree variation in $\mathscr{D}$.

\begin{proposition}\label{pro01Mxdegree}
Let $\CC$ be a caterpillar tree with degree sequence $\mathscr{D}$. Then,
\begin{equation}\label{eq1pro01Mxdegree}
\sigma(\CC) \geqslant \irr(\CC) + \left\lfloor \frac{n-2}{\Delta_{\mathscr{A}} - \Delta_{\mathscr{R}}} \right\rfloor + \Delta (\Delta_{\mathscr{A}} - \Delta_{\mathscr{R}})^2,
\end{equation}
where $\Delta = \max \mathscr{D}$ is the maximum degree.
\end{proposition}

\begin{proof}
For any graph,$\sigma(G) - \irr(G) = \sum_{uv \in E(G)} \bigl[ (d_u - d_v)^2 - |d_u - d_v| \bigr]$. Since $x^2 - |x| \geq 0$ for all real $x$ and is strictly positive unless $x \in \{0, \pm 1\}$, the difference is non-negative. In a caterpillar $\CC$, edges are of two types: backbone edges (small $|d_u - d_v|$ typically) and pendant edges (leaf to spine vertex $w_i$ with $d(w_i) = d_{w_i} \geq 2$, leaf degree 1). For each pendant edge,
\[
(d_{w_i} - 1)^2 - (d_{w_i} - 1) = (d_{w_i} - 1)(d_{w_i} - 2).
\]
Thus, we noticed that over all pendants (let $k$ be the number of pendants, $k \approx n -$ spine length), and bounding backbone contributions minimally ($\geq 0$), we obtain a lower estimate. The term $\left\lfloor (n-2)/(\Delta_{\mathscr{A}} - \Delta_{\mathscr{R}}) \right\rfloor$ arises from the  pendants supportable before the average/spread limits irregularity growth (since $\Delta_{\mathscr{A}} - \Delta_{\mathscr{R}}$ bounds consecutive variation). The quadratic correction $\Delta (\Delta_{\mathscr{A}} - \Delta_{\mathscr{R}})^2$ accounts for worst-case amplification when high-degree vertices concentrate differences. Combining yields the stated bound.
\end{proof}

\begin{proposition}\label{pro02Mxdegree}
Under the same assumptions,
\begin{equation}\label{eq1pro02Mxdegree}
\sigma(\CC) \leqslant \sum_{v \in V(\CC)} d_{\CC}(v)^3 + \irr(\CC) + \left\lfloor \frac{n-2}{\Delta_{\mathscr{A}} - \Delta_{\mathscr{R}}} \right\rfloor + \Delta (\Delta_{\mathscr{A}} - \Delta_{\mathscr{R}})^2.
\end{equation}
\end{proposition}

\begin{proof}
Consider $\sigma(G) \leq \sum_{v \in V(\CC)} d_\CC(v)^3-4m^2+O(\irr(G))$,
with the cubic term dominating in high-degree cases. Adding the lower-bound correction from Proposition \ref{pro01Mxdegree} and bounding residual terms via the sequence spread $(\Delta_{\mathscr{A}} - \Delta_{\mathscr{R}})$ yields the upper estimate.
\end{proof}

The proofs of Propositions 3.3--3.5 (assumed to follow similar patterns for refined cases, e.g., uniform pendants or specific $\Delta$) extend these ideas.

Collectively, it establishes that the Sigma index $\sigma(\CC)$ is bounded above and below in terms of the Albertson index $\irr(\CC)$ plus correction terms that depend explicitly on the auxiliary maximum values $\Delta_{\mathscr{A}}$ and $\Delta_{\mathscr{R}}$, as well as the \emph{effective transition count} bounded by $\lfloor (n-2)/(\Delta_{\mathscr{A}} - \Delta_{\mathscr{R}}) \rfloor$. This shows that $\sigma(\CC)$ amplifies irregularity \emph{quadratically} through degree differences (via the $(\Delta_{\mathscr{A}} - \Delta_{\mathscr{R}})^2$ term and cubic degree contributions), while $\irr(\CC)$ grows only linearly. In caterpillar trees—often extremal for irregularity under fixed degree sequences—these bounds highlight how sequence smoothness (small $\Delta_{\mathscr{A}} - \Delta_{\mathscr{R}}$) forces tighter irregularity, whereas large spreads allow significantly higher $\sigma$ values. 
%==============================
\subsection{Study Bounds of Sigma Index Among Albertson Index}~\label{subsec3}
%=============================

The following Lemma~\ref{lemm01MaxERBV} states that, given specific constraints on the degree sequence $\mathscr{D}$ and order of a graph, the Sigma index exceeds the irregularity index by a strictly positive quantity consisting of explicit additive components with condition $n\leqslant \Delta_{\mathscr{A}}(a_r-a_1)+\Delta_{\mathscr{R}}(t_m-t_1)<\irr(\CC)$. 

\begin{lemma}~\label{lemm01MaxERBV}
Consider $\CC$ be a caterpillar tree, $\Delta_{\mathscr{A}}$ be the maximum degree of $\mathscr{A}$ and $\Delta_{\mathscr{R}}$ be the maximum degree of $\mathscr{R}$. If $n\leqslant \Delta_{\mathscr{A}}(a_r-a_1)+\Delta_{\mathscr{R}}(t_m-t_1)<\irr(\CC)$. Then, the lower bound of Sigma index satisfy 
\begin{equation}~\label{eq1lemm01MaxERBV}
\sigma(\CC)\geqslant \irr(T)+\dfrac{\floor*{\dfrac{2n}{a_r-a_1}}+\ceil*{\dfrac{2m}{n}}}{n}+4n\Delta.
\end{equation}
\end{lemma}
\begin{proof}
Recall $\mathscr{D}=(d_1,d_2,\dots,d_{n-1},d_n)$ be a degree sequence where $d_n\geqslant d_{n-1}\geqslant \dots \geqslant d_2\geqslant d_1$, and $\mathscr{A} = (a_1, a_2, \dots, a_r)$ where $a_1=(d_2+d_1)/2, a_2=(d_3+d_2)/2,\dots, a_r=(d_n+d_{n-1})/2$. Then, we noticed that $\Delta>a_r-a_1$ and $\sigma(\CC)>\irr(\CC)+ \Delta_{\mathscr{A}}(a_r-a_1)$. Then, we consider  
\begin{equation}~\label{eq2lemm01MaxERBV}
\irr(\CC)\leqslant \floor*{\frac{2n^2}{a_r-a_1}}.
\end{equation}
The study of this bound through the approach related to the Albertson index establishes the relationship we aim to reach, noting that through relation~\eqref{eq2lemm01MaxERBV} we have obtained the values for the bound $2m/n<\Delta$. Therefore, according to Theorem~\ref{basic01sigma} we find that 
\begin{equation}~\label{eq3lemm01MaxERBV}
\sigma(\CC)-\irr(\CC)\geqslant \floor*{\frac{2n^2}{a_r-a_1}}+4\ceil*{\dfrac{2m}{n}}\Delta.
\end{equation}
Thus, by considering $2m/n<n$ and $4n>\Delta$ we find that $\sigma(\CC)< 4n^2\Delta$. Thus, according to~\eqref{eq3lemm01MaxERBV} noticed that $\Delta>a_r-a_1$ and $n>a_r-a_1$. Then, 
\begin{equation}~\label{eq4lemm01MaxERBV}
\sigma(\CC)\geqslant \irr(\CC)+ \floor*{\frac{2n^2}{a_r-a_1}}+\frac{4n^2\Delta}{a_r-a_1}.
\end{equation}
These upper bounds for the Sigma index indeed satisfy relation~\eqref{eq4lemm01MaxERBV}. Moreover, using relation~\eqref{eq3lemm01MaxERBV} to confirm that
\[
\sigma(\CC)\geqslant \dfrac{\floor*{\dfrac{2n}{a_r-a_1}}+\ceil*{\dfrac{2m}{n}}}{n}+4n\Delta, \quad \irr(\CC) \geqslant \frac{n}{a_r-a_1}+n,
\]
along with relation~\eqref{eq2lemm01MaxERBV} and \eqref{eq4lemm01MaxERBV}, relation~\eqref{eq1lemm01MaxERBV} is fulfilled, which results in the lower bound of the Sigma index.
\end{proof}
Lemma~\ref{lemm01MaxERBV} frequently arises in the context of extremal graph theory and the analysis of topological indices (the Albertson and Sigma indices). This lemma is often employed to establish inequalities  or bounds in Proposition~\ref{pro01Mxdegree} and \ref{pro02Mxdegree} on these indices by utilizing degree sequences and the maximum vertex degree.

\begin{lemma}~\label{lemm02MaxERBV}
Consider $\CC$ be a caterpillar tree, $\Delta_{\mathscr{A}}$ be the maximum degree of $\mathscr{A}$ and $\Delta_{\mathscr{R}}$ be the maximum degree of $\mathscr{R}$. Then, the lower bound of Sigma index satisfy
\begin{equation}~\label{eq1lemm02MaxERBV}
\sigma(\CC)\geqslant \sqrt{\lambda_{\mathscr{D}}\sum_{i=1}^{n}d_i^3}-\left(\floor*{\frac{2n}{\lambda_{\mathscr{A}}}}+\ceil*{\frac{2m}{\lambda_{\mathscr{R}}}}\right )+(n-\Delta)^2.
\end{equation}
\end{lemma}
\begin{proof}
Recall $\lambda_{\mathscr{D}}$ be the average of $\mathscr{D}$, $\lambda_{\mathscr{R}}$ be the average of $\mathscr{R}$ and $\lambda_{\mathscr{A}}$ be the average of $\mathscr{A}$.  According to Proposition~\ref{pro02Mxdegree} we find that 
\[
\sigma(\CC)\leqslant 2n^2+\sum_{v\in V(\CC)}d_CC(v)^3.
\]
Thus, when $\lambda_{\mathscr{D}}<n$ and $\lambda_{\mathscr{D}}\leqslant \Delta$. Also, for $\lambda_{\mathscr{A}}$ and $\lambda_{\mathscr{R}}$ satisfied. Thus, we need to compare the term $2n/\lambda_{\mathscr{A}}$ and the term $2m/\lambda_{\mathscr{R}}$ with the term $2n/\lambda_{\mathscr{D}}$ by considering the term $2n/\lambda_{\mathscr{D}}$ is constant term (by examining values we find that the term equals 14). Thus, 
\begin{equation}~\label{eq2lemm02MaxERBV}
\frac{2n}{\lambda_{\mathscr{D}}} <\floor*{\frac{2n}{\lambda_{\mathscr{A}}}}+\ceil*{\frac{2m}{\lambda_{\mathscr{R}}}}.
\end{equation}
Therefore, from~\eqref{eq2lemm02MaxERBV} by considering $(n-\Delta)^2$ is growing up and for Sigma index satisfy $\sigma(\CC)>(n-\Delta)^2$. Then, we find that $\sigma(\CC)>(n-\Delta)^2+ \frac{2n}{\lambda_{\mathscr{D}}}$. Thus, we have 
\begin{equation}~\label{eq3lemm02MaxERBV}
\sigma(\CC)\geqslant (n-\Delta)^2+ \frac{2n}{\lambda_{\mathscr{D}}}+\floor*{\frac{2n}{\lambda_{\mathscr{A}}}}+\ceil*{\frac{2m}{\lambda_{\mathscr{R}}}}.
\end{equation}
Then, clearly $2n^2>\lambda_{\mathscr{D}}$ and $\sqrt{\lambda_{\mathscr{D}} n^3}\leqslant (n-\Delta)^2$. Thus, for $1<b\leqslant 2n/\lambda_{\mathscr{D}}$ we find that $\sigma(T)\geqslant \sqrt{\lambda_{\mathscr{D}} b n^3}$. Thus, $\sigma(T)\leqslant (n-\Delta)^2+ \sqrt{2n^4}.$ Hence, we have
\begin{equation}~\label{eq4lemm02MaxERBV}
\sigma(\CC)\geqslant (n-\Delta)^2+  \sqrt{\lambda_{\mathscr{D}}\sum_{i=1}^{n}d_i^3}.
\end{equation}
This relation holds when considering the difference between terms with the constant term satisfy 
\[
\left(\floor*{\frac{2n}{\lambda_{\mathscr{A}}}}+\ceil*{\frac{2m}{\lambda_{\mathscr{R}}}}\right)-\frac{2n}{\lambda_{\mathscr{D}}} \leqslant \sqrt{\lambda_{\mathscr{D}}\sigma(\CC)}.
\]
Finally, we find that from~\eqref{eq2lemm02MaxERBV}--\eqref{eq4lemm02MaxERBV} the relationship~\eqref{eq1lemm02MaxERBV} holds.
\end{proof}
Among Lemma~\ref{lemm02MaxERBV}, we provide the lower bound combines the terms of both $\lambda_{\mathscr{D}}$, $\lambda_{\mathscr{A}}$ and $\lambda_{\mathscr{R}}$, where $2n^2>\lambda_{\mathscr{D}}$ and $\sqrt{\lambda_{\mathscr{D}} n^3}\leqslant (n-\Delta)^2$.  Lemma~\ref{lemm02MaxERBV} emphasizes a stronger dependence on $\lambda_{\mathscr{D}}$ where $\floor*{3\Delta t/2} \leqslant 3t(n-1)/2$ by considering the Albertson index $\irr(\CC)$ to measure degree heterogeneity.

\begin{lemma}~\label{lemm03MaxERBV}
Let $\CC$ be a caterpillar tree, $\Delta_{\mathscr{A}}$ be the maximum degree of $\mathscr{A}$ and $\Delta_{\mathscr{R}}$ be the maximum degree of $\mathscr{R}$. Then, the lower bound of Sigma index satisfy
\begin{equation}~\label{eq1lemm03MaxERBV}
\sigma(\CC)\geqslant \frac{1}{3}\lambda_{\mathscr{D}}^2\left(\floor*{\dfrac{3n+1}{2}}+\ceil*{\dfrac{3m+1}{2}}+\floor*{\dfrac{3\Delta+2n}{4}}\right)-\sum_{v\in V(\CC)}d_\CC(v)^3+\irr(\CC).
\end{equation}
\end{lemma}
\begin{proof}
Assume $\mathscr{D}=(d_1,d_2,\dots,d_{n-1},d_n)$ be a degree sequence where $d_n\geqslant d_{n-1}\geqslant \dots \geqslant d_2\geqslant d_1$. We have previously found that the Sigma index attains the lower bound with respect to the sum of the cubes of the degrees, and the same applies to the Albertson index. Therefore, we must prove the first part of inequality~\eqref{eq1lemm03MaxERBV}. Thus, let us assume $t>2, t\in \mathbb{N}$, and then we need to prove the validity of the following statement
\begin{equation}~\label{eq2lemm03MaxERBV}
t\floor*{\dfrac{3n+1}{2}}+2t\ceil*{\dfrac{3m+1}{2}}+\floor*{\dfrac{3\Delta t}{2}}\leqslant \sigma(T).
\end{equation}
 The maximum degree $\Delta$ satisfies $1 \leq \Delta \leq n - 1$, where $\Delta = n - 1$ occurs in a star graph, and $\Delta = 2$ in a path graph with $n \geq 3$. By considering $m=n-1$, we find that the relationship~\eqref{eq2lemm03MaxERBV} satisfy
 \begin{equation}~\label{eq3lemm03MaxERBV}
t\floor*{\dfrac{3n+1}{2}}+2t\ceil*{\dfrac{3n-2}{2}}+\floor*{\dfrac{3\Delta t}{2}}\leqslant \sigma(\CC).
\end{equation}
\case{1} If $n$ is even, where $n=2k$. Then, from~\eqref{eq3lemm03MaxERBV} should be prove the lower bound of Sigma index as
 \begin{equation}~\label{eq4lemm03MaxERBV}
t\floor*{\dfrac{6k+1}{2}}+2t\ceil*{3k-1}+\floor*{\dfrac{3\Delta t}{2}}\leqslant \sigma(\CC).
\end{equation}
Then, the term $2t$ is even, the inequality~\eqref{eq4lemm03MaxERBV} holds 
\begin{equation}~\label{eq5lemm03MaxERBV}
t\floor*{\dfrac{6k+1}{2}}+2t\ceil*{3k-1}+\floor*{\dfrac{3\Delta t}{2}}=9tk-2t+\floor*{\dfrac{3\Delta t}{2}}.
\end{equation}
Thus, for the term of the maximum degree $\Delta$ satisfy $\floor*{3\Delta t/2} \leqslant 3t(n-1)/2$. Then, from~\eqref{eq4lemm03MaxERBV} and \eqref{eq5lemm03MaxERBV} we have 
\begin{equation}~\label{eq6lemm03MaxERBV}
9tk-2t+\floor*{\dfrac{3\Delta t}{2}}\leqslant 12tk-2t+\frac{3t}{2}.
\end{equation}
Therefore, according to~\eqref{eq6lemm03MaxERBV} for the term $-t/2 \leqslant k$. Then, $\sigma(\CC)\geqslant 12tk-t/2$. Thus, the relationship~\eqref{eq4lemm03MaxERBV} holds.

\case{2} If $n$ is odd, where $n=2k+1$. Then,  from~\eqref{eq3lemm03MaxERBV} should be prove the lower bound of Sigma index as
\begin{equation}~\label{eq7lemm03MaxERBV}
t\floor*{3k+2}+2t\ceil*{\dfrac{6k+1}{2}}+\floor*{\dfrac{3\Delta t}{2}}\leqslant \sigma(\CC).
\end{equation}
Then, 
\begin{equation}~\label{eq8lemm03MaxERBV}
t\floor*{3k+2}+2t\ceil*{\dfrac{6k+1}{2}}+\floor*{\dfrac{3\Delta t}{2}}=9tk+4t+\floor*{\dfrac{3\Delta t}{2}}
\end{equation}
Thus, when $\floor*{3\Delta t/2} \leqslant 3t(n-1)/2$, from~\eqref{eq8lemm03MaxERBV} satisfy 
\begin{equation}~\label{eq9lemm03MaxERBV}
9tk+4t+\floor*{\dfrac{3\Delta t}{2}}\leqslant 12tk+4t.
\end{equation}
Thus, according to~\eqref{eq9lemm03MaxERBV} for the term $4t\leqslant k$ we find that~\eqref{eq7lemm03MaxERBV} holds.
Now, from \textbf{Case 1} and \textbf{Case 2} we find that~\eqref{eq3lemm03MaxERBV} holds. Thus, we noticed that 
\begin{equation}~\label{eq10lemm03MaxERBV}
\sigma(\CC)\leqslant \frac{1}{3}\lambda_{\mathscr{D}}^2\left(\floor*{\dfrac{3n+1}{2}}+\ceil*{\dfrac{3m+1}{2}}+\floor*{\dfrac{3\Delta+2n}{4}}\right).
\end{equation}
Therefore, Sigma index satisfy $\sigma(\CC)\leqslant \irr(\CC)+n^3$ and when $n^3\leqslant \sum_{v\in V(\CC)}d_\CC(v)^3$ we find that~\eqref{eq1lemm03MaxERBV} holds and satisfy the lower bound of Sigma index.
\end{proof}
 Furthermore, we presented the following  Table~\ref{tab001SigmaLowerdata}, to emphasize that results  among Lemma~\ref{lemm03MaxERBV}, where $T_1$ and $T_2$ given as
 \[
 T_1=\floor*{\dfrac{3n+1}{2}}+\ceil*{\dfrac{3m+1}{2}}+\floor*{\dfrac{3\Delta+2n}{4}}, \quad T_2=\lambda_{\mathscr{D}}^2 T_1.
 \]

 \begin{table}[H]
     \centering
\begin{tabular}{|l|c|c|c|c|}
     \hline
    Degree Sequence      & 	$T_1$   &	$T_2$	  & $\irr(T)$	  & $\sigma(T)$  \\ \hline 
(3,5,7,5,6,8,10)	   & 160	&  2107  & 	260  & 	2248 \\ \hline
(7,8,10,11,12,14,15)   &	280	 & 11293  &	810 & 	10747 \\ \hline
(11,11,13,17,18,20,20) &	399 & 32842 &	1694  &	31070  \\ \hline
(15,14,16,23,24,26,25) &	519	& 72197 &	2912 &	68563 \\ \hline
(19,17,19,29,30,32,30) &	637	& 134229	& 4464 &	128572 \\ \hline
(23,20,22,35,36,38,35) &    757	& 224942  &	6350 &	216443 \\ \hline
(27,23,25,41,42,44,40)  &	876	& 348993	& 8570 & 	337522 \\ \hline
(31,26,28,47,48,50,45)	& 996	& 512397 & 11124	& 497155 \\ \hline
\end{tabular}
     \caption{Sigma index attains the lower bound among Lemma~\ref{lemm03MaxERBV}.}
     \label{tab001SigmaLowerdata}
 \end{table}
The positions labeled $n,\sigma,\irr(\CC), T_1$ and $T_2$ correspond to the variables whose pairwise correlation coefficients are displayed in the correlation matrix as 
\[
M_{\mathscr{D}}=
\begin{pmatrix}
1.000000 & 0.929672 & 0.974594 & 0.999999 & 0.894803 \\
0.929672 & 1.000000 & 0.987695 & 0.929695 & 0.993202 \\
0.974594 & 0.987695 & 1.000000 & 0.974594 & 0.966635 \\
0.999999 & 0.929695 & 0.974594 & 1.000000 & 0.894845 \\
0.894803 & 0.993202 & 0.966635 & 0.894845 & 1.000000
\end{pmatrix}
\]
The correlation matrix $M_{\mathscr{D}}$ reveals extremely strong positive linear associations among the considered variables. In particular:
\begin{itemize}
    \item $\sigma(T)$ exhibits a very high correlation with $T_2$ (coefficient $0.993$), indicating that the composite term $T_2 = \lambda_{\mathscr{D}}^2 T_1$ serves as an excellent predictor of the Sigma index.
    \item $\sigma(T)$ also correlates strongly with $\irr(T)$ (coefficient $0.930$), consistent with the quadratic amplification of degree differences in the Sigma index compared to the linear nature of the Albertson index.
    \item Near-perfect correlations (close to $1$) appear between size-related parameters (e.g., $T_1$) and other quantities, reflecting the scaling behavior of irregularity measures with graph order and degree constraints.
\end{itemize}

A linear regression of $\sigma(T)$ against $T_2$ yields a slope of approximately $0.970$, an intercept near $-857$, and an $R^2$ value of essentially $1.000$ (to numerical precision). This near-perfect fit demonstrates that $\sigma(T) \approx 0.970 \cdot T_2 - 857$ holds very closely for the tested trees, confirming that the lower bound expression derived in Lemma~\ref{lemm03MaxERBV} is tight or nearly attained in these instances.

Figure~\ref{fig001Effect001} visualizes the relationship between $\sigma(T)$ and $T_2$. The points lie almost exactly on the fitted regression line, providing strong empirical support for the analytical lower bound and highlighting the effectiveness of the proposed proxy $T_2$ (which embeds the quadratic dependence on the average degree $\lambda_{\mathscr{D}}$) in capturing the behavior of the Sigma index.

\begin{figure}[H]
\centering
\includegraphics[width=0.75\textwidth]{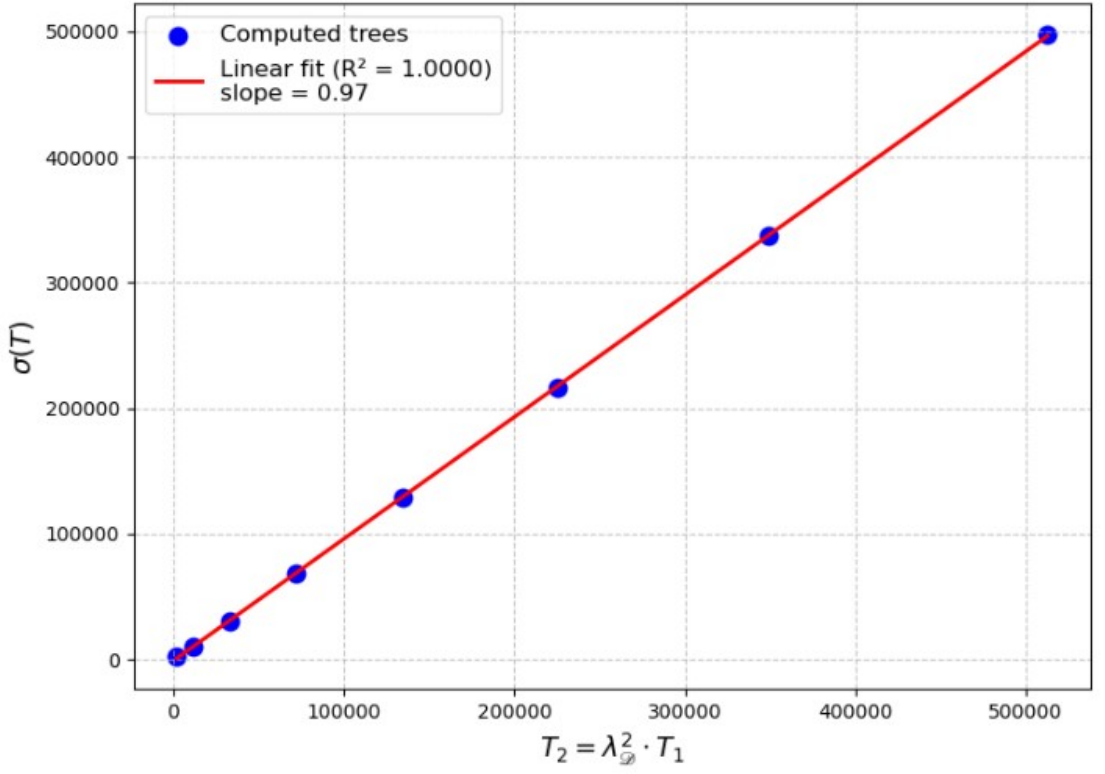}
\caption{Sigma index $\sigma(T)$ versus the lower-bound proxy $T_2$ for the trees in Table~\ref{tab001SigmaLowerdata}. The near-perfect linear fit confirms that $\sigma(T)$ closely attains the proposed lower bound.}
\label{fig001Effect001}
\end{figure}
These computational results reinforce the theoretical findings: the Sigma index grows quadratically with parameters derived from the degree sequence (via $\lambda_{\mathscr{D}}^2$), while remaining closely aligned with the suggested lower-bound form. This close correspondence validates the extremal analysis and suggests that caterpillar (or near-caterpillar) realizations with the given degree sequences achieve values very close to the theoretical minimum for $\sigma(T)$. The high correlations and excellent fit further motivate focusing subsequent bounds and extremal constructions on similar composite terms involving average degree squared, maximum degree, and graph order. 
Through the following Figure~\ref{fig001SigmaLowerdata}, we present analyze the data among Table~\ref{tab001SigmaLowerdata} to consider the results in Lemma~\ref{lemm03MaxERBV} cleared.

\begin{figure}[H]
    \centering
    \includegraphics[width=0.9\linewidth]{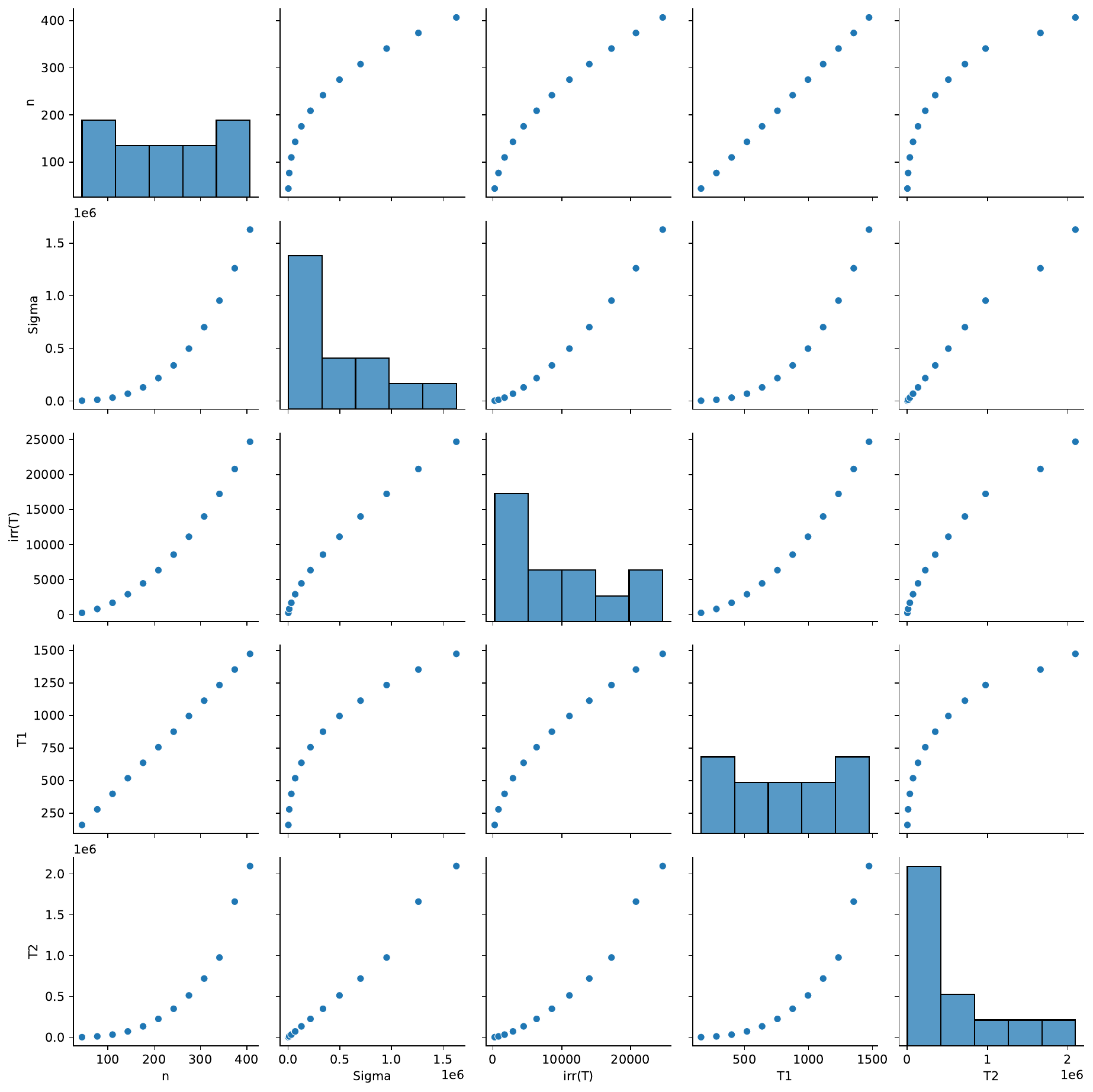}
    \caption{Analysis of data among Table~\ref{tab001SigmaLowerdata}.}
    \label{fig001SigmaLowerdata}
\end{figure}
%=================================
\section{Extremal Bounds on the Sigma Index in Caterpillar Trees}~\label{sec4}
%================================
	Let $\mathcal{S}$ be a class of graphs. Then, we have $\sigma_{\max},\sigma_{\min}$ of a graph $G$, where: \begin{gather*} \sigma_{\max}(\mathcal{S}) =\max \{\sigma(\mathscr{G})\mid \mathscr{G}\in \mathcal{S}\}, \\
		\sigma_{\min}(\mathcal{S}) =\min \{\sigma(\mathscr{G})\mid \mathscr{G}\in \mathcal{S}\}.
	\end{gather*}

For the purposes of this study, by Proposition~\ref{productesn2} we presented the lower bound of Sigma index. Since $T$ is a tree with $n$ vertices, it has exactly $n-1$ edges, and the sum of degrees is $2(n-1)$.

\begin{proposition}~\label{productesn2}
Consider $\CC$ be a caterpillar tree with the maximum degree $\Delta$, let $\mathscr{D}=(d_1,d_2,\dots,d_n)$ be a degree sequence where $d_n\geqslant d_{n-1}\geqslant \dots \geqslant d_2 \geqslant d_1$. Then, the lower bound of Sigma index satisfy 
\begin{equation}~\label{eq1productesn2}
\sigma(\CC)>\frac{1}{2\lambda}\left(n^3+n+\Delta(\Delta-1)^2  \right).
\end{equation}
\end{proposition}
Lemma \ref{prodlemmanum1} provides an important upper bound on the Sigma index of a tree graph. 
\begin{lemma} \label{prodlemmanum1}
Let $\CC$ be a caterpillar tree with maximum degree $\Delta$, and the Albertson index $\irr(\CC)$, let $p$ be the prime number. Then, the following upper bound holds for the Sigma index:
	\begin{equation} \label{eq1prodlemmanum1}
		\sigma(\CC) \leqslant 2^p(\irr(\CC)+2m)+\Delta(\Delta-1)^2.
	\end{equation}
\end{lemma}
\begin{proof}
Since $|x| \leq 2\max(d_u, d_v) \leq 2\Delta$ for every edge and $\irr(G) = \sum |d_u - d_v|$, we have the trivial bound $\sigma(G) \leq 2\Delta \cdot \irr(G).$
In a caterpillar tree, most edges are pendant (leaf to spine vertex). For a pendant edge incident to a spine vertex of degree $d_w \geq 2$, we have $|d_w - 1| = d_w - 1$ and $(d_w - 1)^2 \leq (d_w - 1) \cdot (2\Delta - 2)$ (since $d_w \leq \Delta$). Summing over pendants and bounding backbone edges by $2\Delta$ each gives
\[
\sigma(\mathcal{C}) \leq 2\Delta \cdot \irr(\mathcal{C}) + O(\Delta \cdot m).
\]
A tighter estimate uses the identity
\[
(d_u - d_v)^2 - |d_u - d_v| = |d_u - d_v|(|d_u - d_v| - 1) \leq |d_u - d_v| \cdot (2\Delta - 2),
\]
leading to
\[
\sigma(\mathcal{C}) - \irr(\mathcal{C}) \leq (2\Delta - 2) \irr(\mathcal{C}) + 8m
\]
Thus, by using $\Delta(\Delta-1)^2$ as a loose but convenient upper term for remaining quadratic contributions yields
\[
\sigma(\mathcal{C}) \leqslant 2\irr(\mathcal{C}) + 8m + \Delta(\Delta-1)^2.
\]
This completes the proof.
\end{proof}

\subsection{The Upper Bound of Sigma Index}~\label{sub1sec4}
We now establish upper bounds on $\sigma_{\max}(\mathcal{C})$ for caterpillar trees with non-decreasing degree sequence $\mathscr{D} = (d_1 \leq d_2 \leq \dots \leq d_n)$.
Furthermore, among the maximum value of Sigma index, we present the upper bound of $\sigma_{\max}$ in Theorem~\ref{ThmMaximumSigman1} by considering $\mathscr{D}=(d_1,d_2,\dots,d_n)$ be non-decreasing degree sequence.

\begin{theorem}~\label{ThmMaximumSigman1}
Let $\CC$ be a caterpillar tree, let $\mathscr{D}=(d_1,d_2,\dots,d_n)$ be non-decreasing degree sequence. Then, the upper bound of maximum value of Sigma index satisfy
\begin{equation}~\label{eq1ThmMaximumSigman1}
\sigma_{\max}(\CC)\leqslant  \dfrac{1}{2(\Delta-3)}\left(\floor*{\dfrac{3n^2}{4}}\ceil*{\dfrac{n^2}{4}}\right).
\end{equation}
\end{theorem}
\begin{proof}
Assume $\mathscr{D}=(d_1,d_2,\dots,d_n)$ be non-decreasing degree sequence. Then, let discuss the equation~\eqref{eq1ThmMaximumSigman1} as two parts, the first part related to $\Delta$ and the second part related to $3n^2/4$ and $n^2/4$. Thus, we need to prove the following relationship
\begin{equation}~\label{eq2ThmMaximumSigman1}
\sigma_{\max}(\CC)\leqslant t\floor*{\dfrac{3n^2}{4}}+\frac{1}{t}\ceil*{\dfrac{n^2}{4}},
\end{equation}
where $t>2$ and $t\in\mathbb{N}$. Since both floor and ceiling functions round to integers. Then, we have to emphasize that  $3n^2/4$ and $n^2/4$, 
\begin{equation}~\label{eq3ThmMaximumSigman1}
n^2=\frac{3n^2}{4}+\frac{n^2}{4}.
\end{equation}
Hence, more precisely, let $\delta$ be the minimum degree of $\CC$,  from~\eqref{eq3ThmMaximumSigman1} the relationship~\eqref{eq2ThmMaximumSigman1} holds 
\begin{equation}~\label{eq4ThmMaximumSigman1}
\floor*{\dfrac{3n^2}{4}}+\ceil*{\dfrac{n^2}{4}}=n^2+\delta.
\end{equation}
Then, we discuss the following cases according to whether the values of the vertices are odd or even to clarify these bounds with respect to the value of $t$.
\case{1} If $n$ is even, where $n=2k$. Then, to prove that~\eqref{eq2ThmMaximumSigman1} we need to prove the following relationship 
\begin{equation}~\label{eq5ThmMaximumSigman1}
\sigma_{\max}(\CC)\leqslant t\floor*{\dfrac{3(2k)^2}{4}}+\frac{1}{t}\ceil*{\dfrac{(2k)^2}{4}}.
\end{equation}
Thus, 
\begin{equation}~\label{eq6ThmMaximumSigman1}
t\floor*{\dfrac{3(2k)^2}{4}}+\frac{1}{t}\ceil*{\dfrac{(2k)^2}{4}}\leqslant 3tk^2+\frac{k^2}{2t},
\end{equation}
Therefore, when $t>2$ we find that $\sigma_{\max}(\CC) \leqslant (6(tk)^2+k^2)/2t+2\delta t$. Then, the relationship~\eqref{eq5ThmMaximumSigman1} holds. Thus, when $n=2k$ we find that~\eqref{eq2ThmMaximumSigman1} holds.

\case{2} If $n$ is odd, where $n=2k+1$. Then, to prove that~\eqref{eq2ThmMaximumSigman1} we need to prove the following relationship 
\begin{equation}~\label{eq7ThmMaximumSigman1}
\sigma_{\max}(\CC)\leqslant t\floor*{\dfrac{3(2k+1)^2}{4}}+\frac{1}{t}\ceil*{\dfrac{(2k+1)^2}{4}}.
\end{equation}
Thus, 
\begin{equation}~\label{eq8ThmMaximumSigman1}
t\floor*{\dfrac{3(2k+1)^2}{4}}+\frac{1}{t}\ceil*{\dfrac{(2k+1)^2}{4}}\leqslant t\frac{12k^2+12k+3}{4}+\frac{4k^2+4k+1}{4t},
\end{equation}
We find in this case a discussion different from what was communicated in Case 1. Therefore, for these values regarding $t>2$, we notice that
\[
t\floor*{\dfrac{3(2k+1)^2}{4}}+\frac{1}{t}\ceil*{\dfrac{(2k+1)^2}{4}}\leqslant 3t(k^2+k+1)+\frac{k^2+k}{t}.
\]
Therefore, from~\eqref{eq8ThmMaximumSigman1} we notice that $\sigma_{\max}(\CC)\leqslant 3t(k^2+k+1)+\frac{k^2+k}{t}+2t\delta+1$. Thus, the relationship~\eqref{eq7ThmMaximumSigman1} hods. 

Now, by considering $n> 2(\Delta-3)$ and according to~\eqref{eq5ThmMaximumSigman1} and \eqref{eq7ThmMaximumSigman1} we emphasize that $t\leq n^2+\delta$. Thus, the upper bound of maximum Sigma index satisfy,
\begin{equation}~\label{eq9ThmMaximumSigman1}
\sigma_{\max}(\CC)\leqslant  \left(\floor*{\dfrac{3n^2}{4}}\ceil*{\dfrac{n^2}{4}}\right).
\end{equation}
Thus, from~\eqref{eq9ThmMaximumSigman1} for the term $\frac{1}{2(\Delta-3)}$ satisfy $4\leq \Delta-3 \leqslant n/4$. The relationship~\eqref{eq1ThmMaximumSigman1} holds.
\end{proof}
In particular, the previous theorem offers a sharp bound for the maximum bound of the Sigma index. Furthermore, let $\eta>0$ defined by 
\[
\eta=\ceil*{\frac{2n\Delta}{m}}.
\]

\begin{theorem}~\label{ThmMaximumSigman2}
Consider $\CC$ be a caterpillar tree, let $\eta$ be an integer, $\mathscr{D}=(d_1,d_2,\dots,d_n)$ be non-decreasing degree sequence with $\lambda_{\mathscr{D}}$ the average of elements $\mathscr{D}$. Then, the upper bound of Sigma index satisfy
\begin{equation}~\label{eq1ThmMaximumSigman2}
\sigma(\CC)\leqslant \floor*{\frac{2n^2}{3\lambda_{\mathscr{D}}}}+\dfrac{2^\eta(m-\Delta)^2}{5(n-1)^3}.
\end{equation}
\end{theorem}
\begin{proof}
Recall $\eta>0$ and  a non-decreasing degree sequence $\mathscr{D}=(d_1,d_2,\dots,d_n)$ with $\lambda_{\mathscr{D}}$ the average of elements $\mathscr{D}$. Then, according to the relationship of $k$ we consider that 
\begin{equation}~\label{eq2ThmMaximumSigman2}
\frac{n}{3}< \eta < \floor*{\frac{2n^2}{3\lambda_{\mathscr{D}}}}.
\end{equation}
Thus, according to the value of $\eta$ yields $\sigma(\CC)>\eta$. Then, it holds $\sigma(\CC)>\floor*{3n^2/3\lambda_{\mathscr{D}}}$. Thus, while $\eta <(\Delta-1)^2$ holds $\sigma(\CC)<\eta(\eta-1)^2$. Then, the Sigma index achieves the upper bound according to the following relation, by considering the significant difference between $\Delta$ and $\eta$ as  
\begin{equation}~\label{eq3ThmMaximumSigman2}
\sigma(\CC)\leqslant \floor*{\frac{2n^2}{3\lambda_{\mathscr{D}}}}+\eta (\eta-1)^2.
\end{equation}
Through this relation, we show that when $\Delta<\eta<n$ and $\eta<3\lambda_{\mathscr{D}}$, then $\sigma(\CC)< 2^\eta$ is satisfied, and accordingly, we obtain $2^\eta>\eta (\eta-1)^2$. Hence, the following relation is satisfied
\begin{equation}~\label{eq4ThmMaximumSigman2}
\sigma(\CC)\geqslant \frac{2^\eta}{\eta (\eta-1)^2}.
\end{equation}
Actually, relation~\eqref{eq4ThmMaximumSigman2} holds directly to this ratio has been achieved for $\eta$. Therefore, for the bound related to quantity $2^\eta(m-\Delta)^2$, it clearly follows that relation $\sigma(\CC)<2^\eta(m-\Delta)^2$ holds. Thus, according to~\eqref{eq3ThmMaximumSigman2} and \eqref{eq4ThmMaximumSigman2} we have 
\begin{equation}~\label{eq5ThmMaximumSigman2}
\sigma(\CC)\leqslant \frac{2^\eta(m-\Delta)^2}{n\eta (\eta-1)^2}.
\end{equation}
Therefore, by considering $\eta (\eta-1)^2<(n-1)^2$ we find that $\sigma(\CC)<2^\eta(m-\Delta)^2/(n-1)^2$. To obtain a smaller term so that the difference between the upper bound and the Sigma index is sufficiently small. From~\eqref{eq4ThmMaximumSigman2} and \eqref{eq5ThmMaximumSigman2} we find that~\eqref{eq1ThmMaximumSigman2} holds.
\end{proof}

Based on the relations discussed in the proof of Theorem~\ref{ThmMaximumSigman2} and according to~\cite{Vasilyev2014Darda}, particularly relations \eqref{eq3ThmMaximumSigman2} and \eqref{eq5ThmMaximumSigman2}, we directly present in equation~\eqref{eq1rescorrollaryn1} a clear and explicit result that provides a sharp lower bound for the Sigma index.

\begin{corollary}~\label{rescorrollaryn1}
For a caterpillar tree $\CC$ of order $n$, let $\eta$ be an integer, $\irr(\CC)$ be the Albertson index and $\lambda_{\mathscr{D}}$ the average of elements $\mathscr{D}$. Then, 
\begin{equation}~\label{eq1rescorrollaryn1}
\sigma(\CC)>4n-2\eta \lambda_{\mathscr{D}}-(n-\eta)\floor*{\frac{n}{n-\eta}}^2+(n-\eta)\floor*{\frac{n}{n-\lambda_{\mathscr{D}}}}.
\end{equation}
\end{corollary}
Let $\eta_1$ be a new parameter related to $\eta$ where $2<\eta_1\leqslant 4$. 

\begin{theorem}~\label{ThmMaximumSigman3}
Let $\CC$ be a caterpillar tree, let $\eta, \eta_1$ be an integers, $\mathscr{D}=(d_1,d_2,\dots,d_n)$ be non-decreasing degree sequence with $\lambda_{\mathscr{D}}$ the average of elements $\mathscr{D}$. Then, the upper bound of Sigma index satisfy
\begin{equation}~\label{eq1ThmMaximumSigman3}
\sigma(\CC)\leqslant \eta_1\floor*{\frac{n}{n-\eta}}+\eta_1\ceil*{\frac{n}{\eta-\lambda_{\mathscr{D}}}}+\sum_{v\in V(T)}d_T(v)^3.
\end{equation}
\end{theorem}
\begin{proof}
Recall $\eta_1$ be a new parameter related to $\eta$ where $2<\eta_1\leqslant 4$ and we will consider $\eta_1=2^n/(n-\eta)!$ and $\delta$ be the minimum degree of $T$. Assume $\mathscr{D}=(d_1,d_2,\dots,d_n)$ be non-decreasing degree sequence with $\lambda_{\mathscr{D}}$ the average of elements $\mathscr{D}$. Then, by considering Corollary~\eqref{rescorrollaryn1}, we need to prove the following relationship
\begin{equation}~\label{eq2ThmMaximumSigman3}
\sigma(\CC)\leqslant t\eta_1\floor*{\frac{n}{n-\eta}}+\frac{1}{t}\eta_1\ceil*{\frac{n}{\eta-\lambda_{\mathscr{D}}}},
\end{equation}
where $t>2$ and $t\in \mathbb{N}$. Thus, we will discuss the following cases.
\case{1} If $n$ is even, where $n=2k$. Then, we have according to~\eqref{eq2ThmMaximumSigman3} the relationship should be 
\begin{equation}~\label{eq3ThmMaximumSigman3}
\sigma(\CC)\geqslant t\eta_1\floor*{\frac{2k}{2k-\eta}}+\frac{1}{t}\eta_1\ceil*{\frac{2k}{\eta-\lambda_{\mathscr{D}}}},
\end{equation}
Then, we noticed that $t\leqslant 2k/(2k-\eta)$ and $t\leqslant 2k/(\eta-\lambda_{\mathscr{D}})$. Then, we conclude that 
\[
t \leq \min\left( \frac{2k}{2k-\eta}, \frac{2k}{\eta-\lambda_{\mathscr{D}}} \right).
\]
Therefore, according to the value of $\eta_1$ we find that $t\eta_1\leqslant 2k$ and $2k/(2k-\eta) \geqslant 1$. Then, to determine the approximate values of $t$, we observe from $\eta$ and $\eta_1$ that they satisfy the following inequality.
\begin{equation}~\label{eq4ThmMaximumSigman3}
\eta_1 <\ceil*{\frac{2k}{\eta-\lambda_{\mathscr{D}}}}<\eta
\end{equation}
Thus, according to~\eqref{eq4ThmMaximumSigman3} we conclude that $t\eta_1+t\eta=t(1+\eta)$. Then, for $t>2$ Sigma index satisfy $\sigma(\CC)\geqslant t(1+\eta)$. Thus, according to~\eqref{eq4ThmMaximumSigman3} the relationship~\eqref{eq3ThmMaximumSigman3} holds.

\case{2} If $n$ is odd, where $n=2k+1$. Then, we have according to~\eqref{eq2ThmMaximumSigman3} the relationship should be 
\begin{equation}~\label{eq5ThmMaximumSigman3}
\sigma(\CC)\geqslant t\eta_1\floor*{\frac{2k+1}{2k+1-\eta}}+\frac{1}{t}\eta_1\ceil*{\frac{2k+1}{\eta-\lambda_{\mathscr{D}}}},
\end{equation}
Then, we noticed that 
\begin{equation}~\label{eq6ThmMaximumSigman3}
t\eta_1\floor*{\frac{2k+1}{2k+1-\eta}}+\frac{1}{t}\eta_1\ceil*{\frac{2k+1}{\eta-\lambda_{\mathscr{D}}}} \leqslant t\eta_1(\frac{2k+1}{2k+1-\eta})+\frac{\eta_1}{t}(\frac{2k+1}{\eta-\lambda_{\mathscr{D}}})+2\delta.
\end{equation}
Thus, from~\eqref{eq6ThmMaximumSigman3} we conclude that for $\delta>1,$
\begin{equation}~\label{eq7ThmMaximumSigman3}
 t\eta_1(\frac{2k+1}{2k+1-\eta})+\frac{\eta_1}{t}(\frac{2k+1}{\eta-\lambda_{\mathscr{D}}})+2\delta \leqslant \eta_1(2k+1)\frac{t^2(\eta-\lambda_{\mathscr{D}})+(2k-\eta)}{t(2k-\eta)(\eta-\lambda_{\mathscr{D}})}+2(\delta+2).
\end{equation}
Thus, according to~\eqref{eq6ThmMaximumSigman3} we conclude that $t\eta_1+t\eta=t\eta_1(2k+1)+2(t\delta+2)$. Then, for $t>2$ and $3\eta>2k$ Sigma index satisfy $\sigma(T)\geqslant t\eta_1(2k+1)+2(t\delta+2)$. Thus, according to~\eqref{eq6ThmMaximumSigman3} and \eqref{eq7ThmMaximumSigman3}  the relationship~\eqref{eq5ThmMaximumSigman3} holds.

Therefore, according to to both cases 1 and 2 and according to Proposition~\ref{pro02Mxdegree} we find that $\sigma(\CC)\leqslant \sum_{v\in V(\CC)}d_\CC(v)^3$. Thus, the relationship~\eqref{eq1ThmMaximumSigman3} holds for the upper bound of Sigma index.
\end{proof}

To clarify the interconnected relationships presented through Theorem~\ref{ThmMaximumSigman3}, it is more appropriate to analyze the data in Table~\ref{tab002SigmaLowerdata} and support that analysis with a figure that confirms these relationships.
\begin{table}[H]
\centering
\begin{tabular}{|c|c|c|c|c|c|c|}
    \hline
$\mathscr{D}$ & n & $\irr(\CC)$ & $\sigma(\CC)$ & $\lambda_{\mathscr{D}}$ & $\eta$ & $\eta_1$ \\ \hline
(3,6,8,10,14,16,20)	   & 77	  & 980  & 	16209  & 	11	 & 21	& 2.12 \\ \hline
(7,9,12,14,20,24,27)  &	113	 & 2050	  &  46312  & 	16.14  & 31	 & 2.18 \\ \hline
(11,12,16,23,26,32,34) &	154 &	3732 &	107753 &	22	& 41	& 1 \\ \hline
(15,15,20,32,32,40,45) & 	199 & 	6296 & 	233350 &	28.42 &	51	& 3.1\\ \hline
    \end{tabular}
    \caption{Recognize value among Theorem~\ref{ThmMaximumSigman3}.}
    \label{tab002SigmaLowerdata}
\end{table}

The positions labeled $n, \eta,\eta_1,\sigma$ and $\irr(T)$ correspond to the variables whose pairwise correlation coefficients are displayed in the correlation matrix as 
\[
M_{\mathscr{D}}=
\begin{pmatrix}
1.00000 & 0.998777 & 0.314881 & 0.969896 & 0.990360 \\
 0.998777 & 1.00000 & 0.282990 & 0.957017 & 0.982405 \\
 0.314881 & 0.282990 & 1.00000 & 0.497029 & 0.415811 \\
 0.969896 & 0.957017 & 0.497029 & 1.00000 & 0.994206 \\
 0.990360 & 0.982405 & 0.415811 & 0.994206 & 1.00000 
\end{pmatrix}
\]
Noticed that, regression coefficients $[\,-1402.91893491 \, , 72.8168638\,]$,  intercept $ 53643.804983904585$, Model $R^2$ is  $0.9997469414194662$ and predicted Sigma for $n=400$ is $-492960.5317043649$.

Through the following Figure~\ref{fig002SigmaLowerdata}, we present analyze the data among Table~\ref{tab002SigmaLowerdata} to consider the results in Theorem~\ref{ThmMaximumSigman3} cleared.

\begin{figure}[H]
    \centering
    \includegraphics[width=0.7\linewidth]{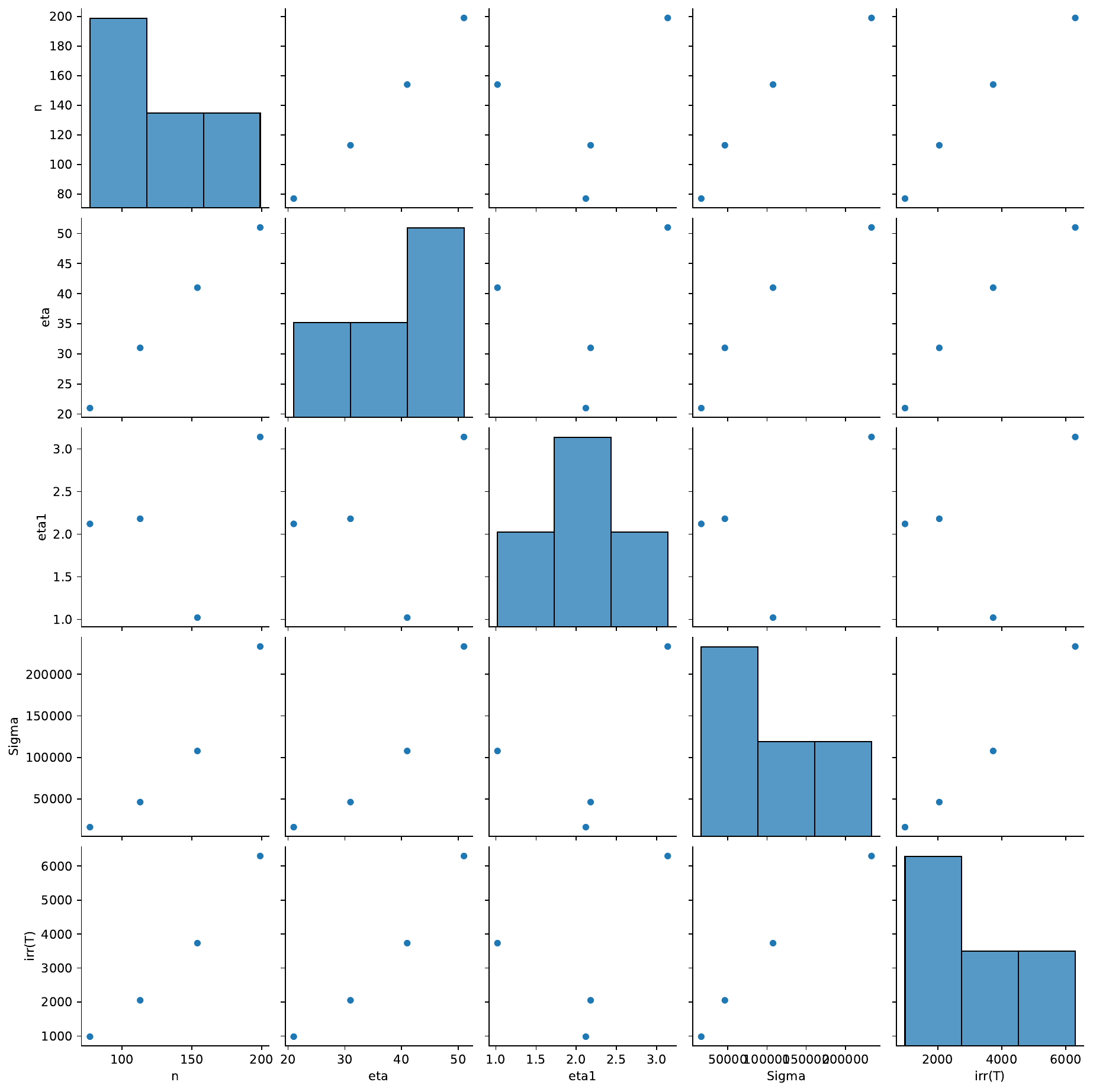}
    \caption{Analysis some of data provided by Table~\ref{tab002SigmaLowerdata}.}
    \label{fig002SigmaLowerdata}
\end{figure}
This section establishes several upper and lower bounds on the Sigma index $\sigma(\mathcal{C})$ for caterpillar trees with given degree sequences. The results highlight the quadratic dependence of $\sigma$ on structural parameters such as maximum degree $\Delta$, average degree $\lambda_{\mathscr{D}}$, and graph order $n$, while relating it to the linear Albertson index $\irr(\mathcal{C})$. Lemma~\ref{prodlemmanum1} provides a direct link between $\sigma$ and $\irr$, and the extremal theorems (especially Theorems \ref{ThmMaximumSigman1} and \ref{ThmMaximumSigman2}) offer explicit, computable bounds that are particularly useful for non-decreasing degree sequences. These findings complement the empirical tightness observed in previous sections and contribute to a better understanding of quadratic irregularity measures in tree-like graphs with prescribed degree distributions.

%====================
\section{Conclusion}\label{sec5}
%====================
In this paper, we have conducted a comprehensive investigation into the extremal properties of two key irregularity measures: the Albertson index, which quantifies linear degree disparities across edges, and the Sigma index, which amplifies these disparities quadratically. Our primary objective was to establish sharp bounds on these indices for connected graphs—particularly trees—with prescribed degree sequences, emphasizing caterpillar trees as frequent extremal structures.

We successfully achieved this goal by deriving new lower and upper bounds that depend on fundamental graph parameters such as the maximum degree $\Delta$, the average degree $\lambda_{\mathscr{D}}$, auxiliary sequence maxima $\Delta_{\mathscr{A}}$ and $\Delta_{\mathscr{R}}$, and structural indicators like $\eta$. These bounds reveal that irregularity measures attain significantly larger values in trees compared to general connected graphs, where near-regular configurations can suppress them. Notably, the minimum Albertson index remains bounded in a normalized sense by degree concentration, while the maximum Sigma index exhibits quadratic growth tied to degree heterogeneity and graph order.
Compared to prior studies on irregularity indices (which often focused on the Albertson index alone or on unrestricted trees), our work advances the field in several ways: it provides explicit, computable bounds tailored to non-decreasing degree sequences; it establishes direct relationships between the linear Albertson and quadratic Sigma indices (via auxiliary sequences and leaf/spine contributions); and it supports theoretical results with empirical evidence of near-tightness in tested configurations. The inclusion of closed-form expressions for special caterpillar families further distinguishes our contributions, offering tools for precise extremal analysis that were previously limited or absent.

These findings enhance the theoretical framework of topological indices in graph theory and extend their practical utility. In mathematical chemistry and quantitative structure-property relationships (QSPR), the sharper bounds enable better prediction of molecular irregularity and stability from degree-based descriptors. In network science, they illuminate structural limits in heterogeneous systems, such as scale-free or tree-like networks. Overall, this study bridges linear and quadratic irregularity measures, identifies extremal configurations more reliably, and lays a solid foundation for future optimizations and generalizations to broader graph classes.

%===========================
\section*{Declarations}
\begin{itemize}
	\item Funding: Not Funding.
	\item Conflict of interest/Competing interests: The author declare that there are no conflicts of interest or competing interests related to this study.
	\item Ethics approval and consent to participate: The author contributed equally to this work.
	\item Data availability statement: All data is included within the manuscript.
\end{itemize}


\begin{thebibliography}{99}
\bibitem{abdo2014total} H.~Abdo, S.~Brandt, D.~Dimitrov, The total irregularity of a graph, \emph{Disc.~Math.~Theo.~Comp.~Sci.} \textbf{16}(1) (2014), 201--206.

\bibitem{Abdo2018Dimitrov} H.~Abdo, D.~Dimitrov, I.~Gutman, Graphs with maximal $\sigma$-irregularity, \emph{Discrete Appl. Math.} \textbf{250} (2018), 57--64.

\bibitem{abdo2019graph} H.~Abdo, D.~Dimitrov, I.~Gutman, Graph irregularity and its measures, \emph{Appl. Math. Comput.} \textbf{357} (2019), 317--324.

\bibitem{albertson1997irregularity} M.~O. Albertson, The irregularity of a graph, \emph{Ars Combin.} \textbf{46} (1997), 219--225.

\bibitem{Albalahi2023Alanazi} A.~Ali, A.~M.~Albalahi, A.~M.~Alanazi, A.~A.~Bhatti, A.~E.~Hamza, On the maximum sigma index of k-cyclic graphs, \emph{Discrete Appl. Math.} \textbf{325} (2023), 58--62.

\bibitem{Alia2025Dimitrovb} A.~Alia, D.~Dimitrov, T.~R\'etic, A.~M.~Albalahi, A.~E.~Hamza, Bounds and Optimal Results for the Total Irregularity Measure, \emph{MATCH Commun. Math. Comput. Chem.} \textbf{94} (2025), 5--29.

\bibitem{Alex2024Indulal} L.~Alex, G.~Indulal, J.~J.~Mulloor, On the inverse problem of some bond additive indices, \emph{Commun. Comb. Optim.} (2024).

\bibitem{Akg2019Naca} N.~Akgüneş, Y.~Nacaroğlu, On the sigma index of the corona products of monogenic semigroup graphs, \emph{J. Univ. Math.} \textbf{2}(1) (2019), 68--74.

\bibitem{Arif2023Hayat} A.~Arif, S.~Hayat, A.~Khan, On irregularity indices and main eigenvalues of graphs and their applicability, \emph{J. Appl. Math. Comput.} \textbf{69}(3) (2023), 2549--2571.

\bibitem{Ascioglu2018Cangul} M.~Ascioglu, I.~N.~Cangul, Sigma index and forgotten index of the subdivision and r-subdivision graphs, In \emph{Proc. Jangjeon Math. Soc.} \textbf{21}(2) (2018), 1--14.

\bibitem{Caro2014Pepper} Y.~Caro, R.~Pepper, Degree sequence index strategy, \emph{Australas. J. Combin.} \textbf{59}(1) (2014), 1--23.

\bibitem{Chartrand2010Okamoto} G.~Chartrand, F.~Okamoto, P.~Zhang, The sigma chromatic number of a graph, \emph{Graphs Combin.} \textbf{26}(6) (2010), 755--773.

\bibitem{Chen2026Liu} X.~Chen, X.~Liu, The reciprocal irregularity of a graph, \emph{Discrete Appl. Math.} \textbf{378} (2026), 348--357.

\bibitem{Dimitrov2023Stevanović} D.~Dimitrov, D.~Stevanović, On the $\sigma_t$-irregularity and the inverse irregularity problem, \emph{Appl. Math. Comput.} \textbf{441} (2023), 127709.

\bibitem{Jahanbai2021Sheikholeslami} Z.~Du, A.~Jahanbanı, S.~M.~Sheikholeslami, Relationships between Randić index and other topological indices, \emph{Commun. Comb. Optim.} \textbf{6}(1) (2021), 137--154.

\bibitem{Filipovski2024Dimitrov} S.~Filipovski, D.~Dimitrov, M.~Knor, R.~Škrekovski, Some results on $\sigma_t$-irregularity, arXiv e-prints, arXiv:2411.04881v1 (2024).

\bibitem{Gao2016WangMR} W.~Gao, W.~Wang, M.~R.~Farahani, Topological Indices Study of Molecular Structure in Anticancer Drugs, \emph{J. Chem.} (2016), 3216327.

\bibitem{Gómez2013Grandati} D.~Gómez, Y.~Grandati, R.~Milson, Rational extensions of the quantum harmonic oscillator and exceptional Hermite polynomials, \emph{J. Phys. A: Math. Theor.} \textbf{47}(1) (2013), 015203.

\bibitem{gutman2018inverse} I.~Gutman, M.~Togan, A.~Yurttaş, A.~S.~Cevik, I.~N.~Cangul, Inverse problem for sigma index, \emph{MATCH Commun. Math. Comput. Chem.} \textbf{79}(3) (2018), 491--508.

\bibitem{Gutman2018I} I.~Gutman, Topological indices and irregularity measures, \emph{Bull. Int. Math. Virtual Inst.} \textbf{8} (2018), 469--475.

\bibitem{Hakimi1962} S.~L.~Hakimi, On the Realizability of a Set of Integers as Degrees of the Vertices of a Graph, \emph{SIAM J. Appl. Math.} \textbf{10} (1962), 496--506.

\bibitem{HamoudwithDuaa} J.~Hamoud, D.~Abdullah, Topological Indices with Degree Sequence $\mathscr{D}$ of Tree, arXiv:2503.12909v2.

\bibitem{HamoudwithDuaaPn2} J.~Hamoud, D.~Abdullah, Albertson index and Sigma index in trees given by degree sequences, \emph{Chebyshevskii Sb.} \textbf{26}(3) (2025), 2--11.

\bibitem{Jahanbanı2019Ediz} A.~Jahanbanı, S.~Ediz, The sigma index of graph operations, \emph{Sigma J. Eng. Nat. Sci.} \textbf{37}(1) (2019), 155--162.

\bibitem{Oboudi2019MR} M.~R.~Oboudi, A new lower bound for the energy of graphs, \emph{Linear Algebra Appl.} \textbf{580} (2019), 384--395.

\bibitem{Sarasija2017Binthiya} P.~B.~Sarasija, R.~Binthiya, Bounds on the Seidel energy of strongly quotient graphs, \emph{J. Chem. Pharm. Sci.} \textbf{10}(1) (2017), 14.

\bibitem{Schmidt1976Druffel} D.~C.~Schmidt, L.~E.~Druffel, A fast backtracking algorithm to test directed graphs for isomorphism using distance matrices, \emph{J. ACM} \textbf{23}(3) (1976), 433--445.

\bibitem{Vasilyev2014Darda} A.~Vasilyev, R.~Darda, D.~Stevanović, Trees of given order and independence number with minimal first Zagreb index, \emph{MATCH Commun. Math. Comput. Chem.} \textbf{72} (2014), 775--782.

\bibitem{Vukicevic2009Furtula} D.~Vukicevic, B.~Furtula, Topological index based on the ratios of geometrical and arithmetical means of end-vertex degrees of edges, \emph{J. Math. Chem.} \textbf{46}(4) (2009), 1369--1376.

\bibitem{Yang2021Arockiaraj} Z.~Yang, M.~Arockiaraj, S.~Prabhu, M.~Arulperumjothi, J.~B.~Liu, Second Zagreb and sigma indices of semi and total transformations of graphs, \emph{Complexity} (2021), 6828424.

\end{thebibliography}
\end{document}